\documentclass{amsart}
\usepackage[british]{babel}
\usepackage{amsthm,amsfonts,enumerate,amscd,url} %,color,showkeys}

\newcommand{\bP}{\mathbb{P}}

\newcommand{\bx}{\mathbf{x}}
\newcommand{\GG}{\mathbb{G}}
\newcommand{\CC}{\mathbb{C}}
\newcommand{\NN}{\mathbb{N}}

\theoremstyle{plain}
\newtheorem{thm}{Theorem}[section]

\newtheorem{prop}[thm]{Proposition}
\newtheorem{cor}[thm]{Corollary}

\newtheorem{lem}[thm]{Lemma}

\theoremstyle{remark}
\newtheorem{re}[thm]{Remark}
\newtheorem{ex}[thm]{Example}

\newtheorem*{acks}{Acknowledgments}

\theoremstyle{definition}
\newtheorem{de}[thm]{Definition}

\numberwithin{equation}{section}

\newcommand{\abs}[1]{\lvert#1\rvert}
\newcommand{\gen}[1]{\langle#1\rangle}

\DeclareMathOperator{\Tan}{Tan}

\DeclareMathOperator{\sym}{Sym}

\DeclareMathOperator{\im}{Im}

\DeclareMathOperator{\rk}{rk}

\DeclareMathOperator{\Hom}{Hom}

\begin{document}

\title{On higher Gauss maps}
\author{Pietro De Poi}
\address{Dipartimento di Matematica e Informatica, Universit\`a degli Studi di Udine, via delle Scienze, 206, 33100 Udine, Italy}
\email{pietro.depoi@uniud.it}
\author{Giovanna Ilardi}
\address{Dipartimento di Matematica e Applicazioni, Universit\`a degli Studi di Napoli ``Federico II'', 80126 Napoli, Italy}
\email{giovanna.ilardi@unina.it}
\thanks{The authors were partially supported by  MIUR,
project ``Geometria delle variet\`a algebriche e dei loro spazi di moduli''}

\date{\today}

\begin{abstract}
We prove that the general fibre of the $i$-th Gauss map has dimension $m$ if and only if at the general point the 
$(i+1)$-th fundamental form consists of cones 
with vertex a fixed $\bP^{m-1}$, extending a known theorem for the usual Gauss map. 
We prove this via a recursive formula for expressing 
higher fundamental forms. We also show some consequences of these results. 
\end{abstract}

\maketitle

\section{Introduction}
Let $V\subset \bP^N$ be a projective embedded algebraic variety and $P\in V$ a point of it; to $P$ it is associated the $i$-th 
osculating space $\widetilde {T}^{(i)}_P(V)\subset \bP^N$ generated by the $(i+1)$-th infinitesimal neighbourhood (see e. g. 
\cite[Example II.3.2.5]{H}); let $d_i$ be its dimension. We recall that 
the \emph{$i$-th Gauss map} $\gamma^i\colon V \dasharrow \GG(d_i,N)$ is the rational map which associates to the general point $P$ 
its   $i$-th osculating space $\widetilde {T}^{(i)}_P(V)$. For $i=1$ this reduces to the usual Gauss map. 

The study of the Gauss map has been rather intense, both in classical 
(see for example \cite{Delp}, \cite{B}, \cite{S}) and modern algebraic geometry (see for example \cite{G-H}, \cite{AG}). 
Instead, the study of higher order Gauss map has been much more scarce; as far as we know, 
we can only cite the classical paper of M. Castellani \cite{C}, reviewed in modern form in \cite{F-I}, \cite{L} and the very recent paper 
\cite{dir}. Closely related results have been obtained by R. Piene on higher dual varieties, see for example \cite{RP}. We note 
moreover that a different  
notion of higher Gauss map has been given by F. Zak: see for example \cite{Z}.

The main basic difference between the first and higher Gauss maps is---as shown in \cite{C}, \cite{F-I} and \cite{RP}---that the fibres of 
the Gauss map are (Zariski open subsets of) linear spaces, while for higher Gauss maps this is not true in general.

In this paper, the study of higher Gauss maps is taken up further: in particular, we analyse the infinitesimal behaviours of 
these maps. Surprisingly enough, it turns out, see Theorem \ref{thm:cono}, that the general fibre of the $i$-th Gauss map has dimension 
$m$ if and only if at the general point the $(i+1)$-th fundamental form consists of cones 
with vertex a fixed $\bP^{m-1}$, as it happens for the usual Gauss map: see \cite[(2.6)]{G-H}. 
This result is obtained by virtue of a recursive formula for expressing the $(i+1)$-th fundamental form in terms of the 
 $i$-th fundamental one: see Lemma \ref{lem:jacco}. This last result is interesting in its own right: for example, we can prove 
immediately as a corollary the (well-known) fact that 
Jacobian system of the $(i+1)$-th fundamental form is contained in the $i$-th fundamental form.
Finally, we show some results that follow from the main result of the paper: in particular, as an example of the consequences of 
the theorem, we study the varieties whose image of 
a higher Gauss map has as image a curve, describing completely their varieties of higher osculating spaces. 

\begin{acks}
We would like to thank the 
anonymous referee for valuable suggestions and remarks. 
\end{acks}

\section{Notation and preliminaries}
We use notation as in \cite{G-H}, \cite{H}, and most of the preliminaries are taken from \cite{De Poi- Di G-I}. 
Let $V \subset \bP^N$ be a projective variety of dimension $k$
over $\CC$ that will be always irreducible. For any point $P\in V$ we use the following notation:
$\widetilde {T}_P(V)\subset \bP^N$
is the embedded tangent projective space to $V$ in $P$ and
$T_P(V)$ is the Zariski tangent space.

As in \cite{G-H}, we abuse notation by identifying the embedded tangent space in $\bP^N$ 
with the affine cone over it in $\CC^{N+1}$. With this convention $T_P(V)\cong \frac{\widetilde T_P(V)}{\CC}$.
We denote by $\GG(t,N)$ the Grassmannian of $t$-planes of $\bP^N$.

To study the behaviour of $V$ in $P$, following \cite{G-H} (and references [2], [6], [7] and [10] therein),
we consider the manifold $\mathcal F (V)$ of frames in $V$. An element of  $\mathcal F (V)$ is a \emph{Darboux frame} centred in $P$. 
This means an $(N+1)$-tuple
\begin{equation*}
\left\{A_0;A_1,\ldots,A_k;\dotsc,A_N\right\}
\end{equation*}
which is a basis of $\CC^{N+1}$ such that, if $\pi:\CC^{N+1}\setminus\{0\}\rightarrow \bP^N$ is the canonical projection,
\begin{align*}
&\pi(A_0)=P,&
&\text{ and }&
&\pi(A_0), \pi(A_1),\ldots,\pi(A_k)\text{ span } \widetilde{T}_P(V).
\end{align*}

Let this frame move in  $\mathcal F (V)$;
then we have the following structure equations (in terms of the restrictions to $V$ of 
the Maurer-Cartan $1$-forms $\omega_i$, $\omega_{i,j}$ on $\mathcal F(\bP^N)$)
for the exterior derivatives of this moving frame

\begin{equation}\label{derivate_frame_eq}
\begin{cases}
 \omega_\mu=0 &\forall \mu>k\\
  d A_0=\sum_{i=0}^k \omega_i A_i \\
  d A_i=\sum_{j=0}^N \omega_{i,j}A_j  &i=1,\dotsc,N\\
  d\omega_j= \sum_{h=0}^k \omega_h \wedge \omega_{h,j} & j=0,\dotsc,k\\
    d\omega_{i,j}= \sum_{h=0}^N \omega_{i,h} \wedge \omega_{h,j}  &i=1,\dotsc,N,\ j=0,\dotsc,N.
\end{cases}
\end{equation}

\begin{re} Geometrically, the frame $\{A_i\}$ defines a coordinate simplex in $\bP^N$.
The $1$-forms $\omega_i, \omega_{i,j}$ give the rotation matrix when the coordinate simplex is infinitesimally displaced; in particular, modulo $A_0$,
as $d A_0\in T^*_P(\bP^N)$ (the cotangent space), the $1$-forms $\omega_1,\dotsc,\omega_k$ give a basis for the cotangent space $T^*_P(V)$,
the corresponding $\pi(A_i)=v_i\in T_P(V)$ give a basis for $T_P(V)$ such that $v_i$ is tangent to the line $\overline{A_0A_i}$,
 and $\omega_{k+1} =\cdots =\omega_N=0$ on $T_P(V)$.
\end{re}

We can define now the $t$-th fundamental form and the
$t$-th osculating space at $P\in V$, for $t\geq 2$, see \cite[\S 1.(b) and (d)]{G-H} and \cite[\S 1]{De Poi- Di G-I}: 

\begin{de}
Let $P\in V$, let $t\geq 2$ be an integer and let $I=(i_1,\ldots,i_k)$ be such that $\abs{I}\leq t$.
The  \emph{$t$-th osculating space} to $V$ at $P$ is the subspace $\tilde T^{(t)}_P(V)\subset \bP^N$ spanned by $A_0$ and by all the
derivatives $\dfrac{d^{\abs{I}}A_0}{d v_1^{i1}\dotsm d v_k^{i_k}}$, where $v_1,\dotsc, v_k$ span $T_P(V)$.

If the point $P\in V$ is general, we will put 
\begin{equation*}
d_t:= \dim(\tilde T^{(t)}_P(V)).
\end{equation*}
\end{de}

\begin{re}
Obviously, the ``expected'' dimension of the $t$-th osculating space at a general point $P$ 
for our $k$-dimensional variety $V\in\bP^N$ is the minimum among  
$N$ and $k_t:=\binom{k+t}{k}-1$; indeed, for the ``general'' variety, $d_t$ has this expected dimension, but, 
as the next example shows, there are exceptions.  
\end{re}

\begin{ex}
Let us consider the so called \emph{Togliatti surface} $S\subset\bP^5$, 
which is a particular projection of the del Pezzo sextic surface embedded in $\bP^6$, see \cite{T} and \cite{I};
 it is the closure of the image of the rational map
\begin{align*}
\psi\colon \bP^2 &\dashrightarrow  \bP^5\\
(x:y:z) &\mapsto (x^2y: xy^2: x^2z: xz^2: y^2z: yz^2).
\end{align*}
For the Togliatti surface we have $d_2=4$:  this can be easily seen from the fact that the parametrisation of $S$ given above, 
restricted to the open affine set $x\neq 0$, 
satisfies the following second order partial differential equation  (classically called \emph{Laplace equation}): 
\begin{equation*}
u^2\frac{\partial^2 \bx}{\partial u^2}+uv\frac{\partial^2 \bx}{\partial u\partial v}+v^2\frac{\partial^2 \bx}{\partial v^2}
-2 u\frac{\partial \bx}{\partial u}-2 v\frac{\partial \bx}{\partial v}+2\bx=0,
\end{equation*}
where $u=\frac{y}{x}$, $v=\frac{z}{x}$, are the affine coordinates, etc. 
\end{ex}

\begin{ex}
It is immediate to see that for the Veronese surface  $V_2\subset \bP^5$ we have $\dim \tilde T^{(2)}_P(V_2)=5$ for every point $ P\in V_2$: 
there is not a hyperplane section of $V_2$ with a triple point. 
\end{ex}

\begin{de}
Let $t\geq 2$ and $V_0\subseteq V$ be the quasi projective variety of points where $\tilde T ^{(t)}_P (V)$ has maximal dimension.
The variety 
\begin{equation*}
{\Tan}^t(V) := { \overline{\bigcup_{P \in V_0} \tilde T^{(t)}_P(V)}}
\end{equation*}
  is called  the \emph{variety of  osculating $t$-spaces to $V$}. 
\end{de}

\begin{re}
Obviously we have
\begin{equation*}
d_t\le d_{t-1}+ \binom{k-1+t}{t}\le \dotsm \le \sum_{i=1}^t\binom{k+i-1}{i}= k_t. 
\end{equation*}
\end{re} 
Following \cite{G-H} %and recalling \eqref{eq_II} 
we give 
\begin{de}\label{ff_def}
The \emph{$t-$th fundamental form} of $ V$ in $P$
 is the linear system $\abs{I^t}$ in the projective space $\bP (T_P(V))\cong \bP^{k-1}$ of hypersurfaces of degree $t$ 
defined symbolically by the equations:
\begin{equation*}
d^t A_0=0.
\end{equation*}
More intrinsically, we write $I^t$ as the map
\begin{equation*}
   I^t\colon  \sym^{(t)} T(V) \rightarrow N^t(V)
\end{equation*}
where $N^t(V)$ is the bundle defined locally as $N^t_P(V):=\displaystyle{\frac{\CC^{N+1}}{{\tilde T}^{(t-1)}_P(V)} }$
and the map $I^t$ is defined locally on each $v\in T_P(V)$ as
\begin{equation*}
     v^t \mapsto \displaystyle{\frac{d^t A_0}{d v^t} }\mod {\tilde T}^{(t-1)}_P(V).
\end{equation*}
\end{de}

\begin{ex}
With an explicit calculation (or from Example \ref{ex:3}) we can show that 
the third fundamental form at the general point $P$ of the Togliatti surface $S\subset \bP^5$
is given by a single cubic in $\bP (T_P(V))\cong \bP^1$, while for a ``general'' surface (i.e. if $d_2=5$) of $\bP^5$ 
the third fundamental form is the empty set; this is obviously true for the Veronese surface $V_2\subset\bP^5$. 
\end{ex}

We will denote the dimension of the $t$-th fundamental form by $\Delta_t$:
\begin{equation*}
\Delta_t:=\dim(\abs{I^t}),
\end{equation*}
i.e. $\abs{I^t}\cong \bP^{\Delta_t}$. 

We recall that (see \cite[Corollary 1.15]{De Poi- Di G-I})

\begin{prop}\label{cor_dimensione ff}
We have that 
\begin{equation*}
d_t=  d_{t-1}+\Delta_t+1,
\end{equation*}
and vice versa: if $d_t=  d_{t-1}+\Delta+1$, then the $t$-th fundamental form has dimension $\Delta$. 
\end{prop}

\begin{ex}\label{ex:3}
For the Togliatti surface $S\subset \bP^5$ we have that $d_3=5$ and $d_2=4$, and therefore $\Delta_3=0$. 

Analogously,  for a ``general'' surface of $\bP^5$, $d_3=d_2=5$, and $\Delta_3=-1$, i.e. $\abs{III}=\emptyset$;  
this is true for example for the Veronese surface $V_2\subset\bP^5$. 
\end{ex}

From now on, we will suppose that our Darboux frame 
\begin{equation}\label{eq:darbu}
\{A_0;A_1,\dotsc,A_k;A_{k+1},\dotsc,A_{d_2};A_{d_2+1},\dotsc,A_{d_s};\dotsc,A_{d_t};\dotsc, A_N\} 
\end{equation}
is such that $A_0,A_1,\dotsc,A_{d_s}$ span ${\tilde T}^{(s)}_P(V)$ for all $ s=1,\dotsc, t$, with $d_1:=k$.

\begin{de}
Let $t\geq 1$. The \emph{$t$-th (projective) Gauss map} is the rational map
\begin{align*}
    \gamma^t\colon &V \dashrightarrow \GG(d_t,\bP^N)\\
    & P \mapsto {\tilde T}^{(t)}_P(V).
\end{align*}
\end{de}

We recall that (see for example \cite[Theorem 1.18]{De Poi- Di G-I})
\begin{thm}
The first differential of $\gamma^t$ at $P$ is the $(t+1)$-th fundamental form at $P$.
\end{thm}

The idea of the proof is the following: 
we have, by the definition of $\gamma^t$, that
\begin{equation*}
    d\gamma^t_P\colon T_P V \dashrightarrow T_{\tilde T^{(t)}_P V} \GG(d_t,\bP^N),
\end{equation*}
and we recall that  $T_{\tilde T^{(t)}_P V} \GG(d_t,\bP^N)\cong \Hom(\tilde T^{(t)}_P V, N^{t+1}_P(V))$; moreover
if we choose a Darboux frame as in \eqref{eq:darbu}, we have that $d A_0\in \tilde T_P V \subset \tilde T^{(t)}_P V$ and
\begin{equation*}
\frac{\tilde T^{(t)}_P V}{\CC A_0}=T^{(t)}_P V
\end{equation*}
and therefore  $d\gamma^t_P\in \Hom (T_P V\otimes T^{(t)}_P V,  N^{t+1}_P(V)) $.

Now, we remark that, in our Darboux frame, we can interpret $\gamma^t$ as
\begin{equation*}
 \gamma^t(P)=A_0\wedge\dotsb \wedge A_{d_t},
\end{equation*}
and therefore by \eqref{derivate_frame_eq},
\begin{align*}
 d\gamma^t_P&\equiv\sum_{\substack{
1\le i \le d_t\\
d_t+1\le j \le N
}}(-1)^{d_t-i+1} \omega _{i,j} A_0\wedge\dotsb\wedge \hat {A_i}\wedge\dotsb \wedge A_{d_t}\wedge A_{j}, & \mod \tilde T^{(t)}_P V;
\end{align*}
now, a basis for $T_P V\otimes T^{(t)}_P V$ can be expressed by $(A_\alpha\otimes A_\mu)_{\substack{\alpha=1,\dotsc,k\\ \mu=1,\dotsc, d_t}}$,
and
\begin{align*}
 d\gamma^t_P(A_\alpha\otimes A_\mu) &=\sum_{
d_t+1\le j \le N} \omega _{\mu,j}(A_\alpha) A_{j}\in  N^{t+1}_P(V)
\end{align*}
on the other hand, for the $(t+1)$-th fundamental form we have
\begin{align*}
\frac{d A_\mu}{d v_\alpha} &\equiv \sum_{\substack{
d_t+1\le j \le N
}} \omega _{\mu,j}(A_\alpha) A_{j} &  \mod \tilde T^{(t)}_P(V).
\end{align*}

%%%%%%%%%%%%%%%%%%%%%%%%%%%%%%%%%%%%%
%%%%%%%%%%%%%%%%%%%%%%%%%%%%%%%%%%%%%
%%%%%%%%%%%%%%%%%%%%%%%%%%%%%%%%%%%%%

%\rosa{Cominciamo a scrivere qui}

\section{Fibres of higher Gauss maps}

We start by writing higher fundamental forms explicitly: 
choose a Darboux frame 
\begin{equation*}
\{A_0;A_1,\dotsc,A_k;A_{k+1},\dotsc,A_{d_2};A_{d_2+1},\dotsc,A_{d_s};\dotsc,A_{d_t};\dotsc, A_N\} 
\end{equation*}
such that $A_0,A_1,\dotsc,A_{d_s}$ span ${\tilde T}^{(s)}_P(V)$ for all $ s=1,\dotsc, t$, with $d_1:=k$, 
${\tilde T}^{(1)}_P(V)={\tilde T}_P(V)$ (and $d_0:=0$).
We use indices $1\le i^{(1)}, i_1^{(1)}, \dotsc, i_h^{(1)}\le k=d_1$,  $d_{s-1}+1\le i^{(s)}\le d_s$ for $s=2,\dotsc t$ 
and $d_t+1\le i^{(t+1)}\le N$.

With these notations, the second fundamental form is given by the quadrics
\begin{equation}\label{eq:secondo}
V^{(2)}_{i^{(\ell)}}=\sum_{i_1^{(1)},i_2^{(1)}}q_{i_1^{(1)},i_2^{(1)};i^{(\ell)}}\omega_{i_1^{(1)}} \omega_{i_2^{(1)}},
\end{equation}
with $\ell>k$, where $q_{i_1^{(1)},i_2^{(1)};i^{(\ell)}}(=q_{i_2^{(1)},i_1^{(1)};i^{(\ell)}})$ are defined by 
\begin{equation}\label{eq:secondobis}
\omega_{i_1^{(1)}, i^{(\ell)}}=\sum_{i_2^{(1)}}q_{i_1^{(1)},i_2^{(1)};i^{(\ell)}} \omega_{i_2^{(1)}},
\end{equation}
which are obtained via the Cartan Lemma from 
\begin{equation*}
0=d\omega_{i^{(\ell)}}=\sum_{i_1^{(1)}}\omega_{i_1^{(1)}}\wedge \omega_{i_1^{(1)}, i^{(\ell)}},
\end{equation*}
since $\omega_{i^{(\ell)}}=0$ on ${\tilde T}_P(V)$. 

The higher fundamental forms can be expressed as in the following 

\begin{lem}\label{lem:jacco}
The  $(s+1)$-th fundamental form is given by the  polynomials 
\begin{equation}\label{eq:cacco}
V^{(s+1)}_{i^{(\ell)}}=\sum_{i_1^{(1)}, \dotsc, i_{s+1}^{(1)}=1}^kq_{i_1^{(1)}, \dotsc, i_{s+1}^{(1)};i^{(\ell)}}\omega_{i_1^{(1)}}\dotsm \omega_{i_{s+1}^{(1)}},
\end{equation}
with $\ell\ge s+1$, which inductively satisfy the relations
\begin{equation}\label{eq:relazioni}
\sum_{i^{(s)}=d_{s-1}+1}^{d_s}q_{i_1^{(1)}, \dotsc, i_s^{(1)};i^{(s)}} \omega_{i^{(s)}, i^{(\ell)}}
=\sum_{i_{s+1}^{(1)}=1}^kq_{i_1^{(1)}, \dotsc, i_{s+1}^{(1)};i^{(\ell)}} \omega_{i_{s+1}^{(1)}},
\end{equation}
where $q_{i_1^{(1)}, \dotsc, i_s^{(1)};i^{(s)}}$ are the coefficients of the $s$-th fundamental form 
(and with the natural symmetries of the indices: 
$q_{\dotsc, i_j^{(1)}, \dotsc, i_k^{(1)}\dotsc;i^{(\ell)}}=q_{\dotsc,i_k^{(1)}, \dotsc, i_j^{(1)},\dotsc;i^{(\ell)}}$), and the basis of the induction is 
the second fundamental form (i.e. $s=1$) for which relations \eqref{eq:cacco} and \eqref{eq:relazioni} are to be read as,  
respectively, \eqref{eq:secondo} and \eqref{eq:secondobis}.
 \end{lem}

\begin{proof}
By the definition of  ${\tilde T}^{(s)}_P(V)$, where $s\in\{1,\dotsc, t-1\}$, we have that 
\begin{equation}\label{eq:dsup1}
d A_{i^{(s-1)}} \equiv 0 \mod  \tilde T^{(s)}_P(V)
\end{equation}
and from  \eqref{derivate_frame_eq} we can write 
\begin{equation*}
d A_{i^{(s-1)}}= \sum_{j=d_{s-1}+1}^N\omega_{i^{(s-1)},j}A_j 
\end{equation*}
from which we deduce, by \eqref{eq:dsup1}
\begin{equation}\label{eq:omega1}
\omega_{i^{(s-1)}, i^{(s+1)}}=\omega_{i^{(s-1)}, i^{(s+2)}}=\dotsb= \omega_{i^{(s-1)}, i^{(t)}}=\omega_{i^{(s-1)}, i^{(t+1)}}=0.
\end{equation}
These equations imply that 
\begin{align*}
0=d\omega_{i^{(s-1)}, i^{(\ell)}}&=\sum_{i^{(s)}}\omega_{i^{(s-1)}, i^{(s)}}\wedge \omega_{i^{(s)}, i^{(\ell)}}& \ell&\ge s+1;
\end{align*}
now, by the inductive hypothesis %(Cartan lemma applied to the $(s-1)$-th pass)
, we have
\begin{equation}\label{eq:s-1}
\sum_{i^{(s-1)}}q_{i_1^{(1)}, \dotsc, i_{s-1}^{(1)};i^{(s-1)}}\omega_{i^{(s-1)}, i^{(\ell)}}=\sum_{i_s^{(1)}}q_{i_1^{(1)}, \dotsc, i_s^{(1)};i^{(\ell)}}\omega_{i_s^{(1)}}
\end{equation}
where $q_{i_1^{(1)}, \dotsc, i_{s-1}^{(1)};i^{(s-1)}}$ are the coefficients of the $(s-1)$-th fundamental form and $q_{i_1^{(1)}, \dotsc, i_s^{(1)};i^{(\ell)}}$ are 
those of the $s$-th fundamental form (with the convention that $q_{i_1^{(1)}, \dotsc, i_s^{(1)};i^{(\ell)}}=1$ if $s\le 1$, and if $s=2$, 
\eqref{eq:s-1} is just \eqref{eq:secondobis} 
i.e. Cartan lemma used to obtain the second fundamental 
form \eqref{eq:secondo}) 
and then, applying relations \eqref{eq:s-1} to (a linear combination of) \eqref{eq:omega1}, we deduce 
\begin{equation*}
\sum_{i^{(s-1)},i^{(s)} }q_{i_1^{(1)}, \dotsc, i_{s-1}^{(1)};i^{(s-1)}}\omega_{i^{(s-1)}, i^{(s)}}\wedge \omega_{i^{(s)}, i^{(\ell)}}
=\sum_{i_s^{(1)},i^{(s)} }q_{i_1^{(1)}, \dotsc, i_s^{(1)};i^{(s)}}\omega_{i_s^{(1)}}\wedge \omega_{i^{(s)}, i^{(\ell)}}
\end{equation*}
then, by Cartan lemma
\begin{equation*}
\sum_{i^{(s)}}q_{i_1^{(1)}, \dotsc, i_s^{(1)};i^{(s)}} \omega_{i^{(s)}, i^{(\ell)}}=\sum_{i_{s+1}^{(1)}}q_{i_1^{(1)}, \dotsc, i_{s+1}^{(1)};i^{(\ell)}} \omega_{i_{s+1}^{(1)}};
\end{equation*}
and the $(s+1)$-th fundamental form is 
\begin{equation*}
V^{(s+1)}_{i^{(\ell)}}=\sum_{i_1^{(1)}, \dotsc, i_{s+1}^{(1)}}q_{i_1^{(1)}, \dotsc, i_{s+1}^{(1)};i^{(\ell)}}\omega_{i_1^{(1)}}\dotsm \omega_{i_{s+1}^{(1)}}.
\end{equation*}
\end{proof}

\begin{re}
Alternatively, we can write---for example writing back relations \eqref{eq:relazioni} in \eqref{eq:cacco} or by a direct computation
\begin{equation*}
V^{(s+1)}_{i^{(\ell)}}=\sum_{i^{(1)}, \dotsc, i^{(s)}}\omega_{i^{(1)}}\omega_{i^{(1)}, i^{(2)}}  \omega_{i^{(1)}, i^{(3)}}\dotsm
\omega_{i^{(s-1)}, i^{(s)}} \omega_{i^{(s)}, i^{(\ell)}},
\end{equation*}
with $\ell\ge s+1$ and the sum varies, as above, for $1\le i^{(1)}\le k=d_1,\dotsc, d_{s-1}+1\le i^{(s)}\le d_s$. 
\end{re}
We recall 

\begin{de}
Let $\Sigma$  be the linear system of dimension $d$ of hypersurfaces of degree $n>1$ in  $\bP^N$ ($N>1$), generated
by the $d+1$ hypersurfaces $f_0=0, \dotsc, f_d=0$. The \emph{Jacobian system} $J(\Sigma)$ of $\Sigma$  is the linear system defined by
the partial derivatives of the forms $f_0, \dotsc, f_d$:
\begin{align*}
J(\Sigma)&:=(\partial f_i/\partial x_j) &i&=0, \dotsc,d;\ j=0,\dotsc,r.
\end{align*}
\end{de}

\begin{re}
Obviously, the Jacobian system $J(\Sigma)$ does not 
depend on the choice of $f_0, \dotsc, f_d$, but only on $\Sigma$. 
\end{re}

\begin{cor}\label{cor:jac}
The $(s+1)$-th fundamental  form  is a linear system of polynomials of degree $s+1$ 
whose Jacobian system is contained in the  $s$-th fundamental  form. 
\end{cor}

\begin{proof}
We have, by Lemma \ref{lem:jacco}
\begin{equation}\label{eq:abel}
\frac{\partial V^{(s+1)}_{i^{(\ell)}}}{\partial \omega_{i_j^{(1)}}}=\sum_{i_1^{(1)}, \dotsc, i_{s}^{(1)}, i^{(s)}}q_{i_1^{(1)}, \dotsc, i_s^{(1)};i^{(\ell)}}\omega_{i_1^{(1)}}\dotsm \omega_{i_{s}^{(1)}}\frac{\partial \omega_{i^{(s)}, i^{(\ell)}}}{\partial \omega_{i_j^{(1)}}}.
\end{equation}

\end{proof}

We are now able to prove the following fact (see \cite[(2.6)]{G-H})

\begin{thm}\label{thm:cono}
The $s$-th Gauss map $\gamma^s$ ($s\ge 1$) has fibres of dimension $m$ if and only if at a general point $P\in V$ the forms of degree $s+1$ of the $(s+1)$-th 
fundamental form $\abs{I^{s+1}}$ are cones over a $\bP^{m-1}\subset \bP T_P(V)$.
\end{thm}

\begin{proof}
We choose a Darboux frame as above, i.e. 
\begin{equation*}
\{A_0;A_1,\dotsc,A_k;A_{k+1},\dotsc,A_{d_2};A_{d_2+1},\dotsc,A_{d_s};\dotsc,A_{d_t};\dotsc,A_{d_{t+1}};\dotsc, A_N\} 
\end{equation*}
such that $A_0,A_1,\dotsc,A_{d_s}$ span ${\tilde T}^{(s)}_P(V)$ and indices 
$d_{s-1}+1\le i^{(s)}\le d_s$ for all $ s=1,\dotsc, t+1$, with $d_1:=k$, etc.

In coordinates, the $s$-th Gauss map  $\gamma^s$  is
\begin{equation*}
 \gamma^s(P)=A_0\wedge\dotsb \wedge A_{d_s},
\end{equation*}
and therefore by \eqref{derivate_frame_eq}, its differential, which is the $(s+1)$-th fundamental form, is given by 
\begin{align*}
 d\gamma^s_P&\equiv\sum_{\substack{
1\le i \le d_s\\
d_s+1\le j \le N
}}(-1)^{d_s-i+1} \omega _{i,j} A_0\wedge\dotsb\wedge \hat {A_i}\wedge\dotsb \wedge A_{d_s}\wedge A_{j}, & \mod \tilde T^{(s)}_P V;
\end{align*}
therefore, since $\gamma^s$ has fibres of dimension $m$ if and only if  $d\gamma^s_P$ has rank $k-m$ (recall that $P$ is a general point)
 $\gamma^s$ has fibres of dimension $m$ if and only if  the space  
\begin{equation*}
U^*:=\gen{\omega _{i,j}}_{\substack{
1\le i \le d_s\\
d_s+1\le j \le N
}}\subset T^*_P(V) 
\end{equation*}
has dimension $k-m$. Dually, this space defines a subspace  $U\subset T_P(V)$ of dimension $m$, defined by the equations
\begin{align*}
\omega _{i^{(h)},i^{(\ell)}}&=0, & h=1,\dotsc, s;\ &\ell\ge s+1. 
\end{align*}
Now, we prove that $V^{(s+1)}_{i^{(\ell)}}$ are cones with vertex $\bP(U)$: 
let us suppose to choose a frame such that $\omega_{m+1},\dotsc,\omega_k$ form a basis for $U^*$,
that is 
\begin{equation}\label{eq:genero}
\gen{\omega _{i,j}}_{\substack{
1\le i \le d_s\\
d_s+1\le j \le N
}}=\gen{\omega_{m+1},\dotsc,\omega_k}.  
\end{equation}
Moreover, we choose indices $1\le j_1^{(1)}, \dotsc, j_h^{(1)}\le m$, and $m+1\le k_1^{(1)}, \dotsc, k_h^{(1)}\le k$. 

From \eqref{eq:relazioni}, the points of $U$ satisfy also 
\begin{equation}\label{eq:relazioni1}
0=\sum_{i^{(h)}}q_{i_1^{(1)}, \dotsc, i_h^{(1)};i^{(h)}} \omega_{i^{(h)}, i^{(\ell)}}=\sum_{i_{h+1}^{(1)}}q_{i_1^{(1)}, \dotsc, i_{h+1}^{(1)};i^{(\ell)}} \omega_{i_{h+1}^{(1)}},
\end{equation}
which implies, for every choice of the indices as above, that 
\begin{equation*}%\label{eq:relazioni2}
q_{i_1^{(1)}, \dotsc, j_{h+1}^{(1)};i^{(\ell)}}=0.
\end{equation*}
From this we infer, from \eqref{eq:cacco}, that 
\begin{equation}
V^{(s+1)}_{i^{(\ell)}}=\sum_{k_1^{(1)}, \dotsc, k_{s+1}^{(1)}}q_{k_1^{(1)}, \dotsc, k_{s+1}^{(1)};i^{(\ell)}}\omega_{k_1^{(1)}}\dotsm \omega_{k_{s+1}^{(1)}},
\end{equation}
that is, $V^{(s+1)}_{i^{(\ell)}}$ is a cone with vertex  $\bP^{m-1}=\bP(U)$.

Vice versa, let us suppose that the polynomials  $V^{(s+1)}_{i^{(\ell)}}$ are singular along a $\bP^{m-1}=\bP(U)$; by choosing our frame, we can suppose 
that $U$ is defined by 
\begin{align*}
\omega _j&=0, & j=m+1,\dotsc, k, 
\end{align*}
or that \eqref{eq:genero} holds. 
Then, we have
\begin{equation}\label{eq:frobenio}
0=d\omega_{i^{(h)}, i^{(\ell)}}=\sum_{i=1}^{d_s}\omega_{i^{(h)}, i}\wedge \omega_{i, i^{(\ell)}}+ 
\sum_{j=d_s+1}^N\omega_{i^{(h)}, j}\wedge \omega_{j, i^{(\ell)}}
\end{equation}
and $\omega_{i, i^{(\ell)}}, \omega_{i^{(h)}, j}\in U^*$ and therefore the Frobenius integrability conditions for $\omega_{m+1},\dotsc,\omega_k$
are satisfied; then, from 
\begin{equation*}
d A_{i^{(s)}}= \sum_{j=d_s+1}^N\omega_{i^{(s)},j}A_j 
\end{equation*}
we conclude that the $s$-th osculating space remains constant along the leaves of the foliation defined by $\omega_{m+1},\dotsc,\omega_k$.

\end{proof}

\begin{re}
As we recalled in the introduction, one can deduce, from Theorem~\ref{thm:cono}, that, for $s=1$ (i.e. the usual Gauss map),
the fibres of $\gamma^1$ are  (Zariski open subsets of) linear spaces: see \cite[(2.10)]{G-H}. 
Moreover, in \cite[\S 3]{G-H} they give a general description of these varieties; we will pursue further this kind of 
description for higher $s$ in Subsection \ref{ss:3}.
\end{re}

\begin{ex}
The first nontrivial example of application of the preceding theorem with $s\ge 2$ is given by surfaces whose second Gauss map $\gamma^2$ 
has as image a curve; 
this case has been studied in \cite{C} and in \cite{F-I}: they proved that the curves which are fibres of  $\gamma^2$ are indeed 
contained in $3$-dimensional projective spaces; by Theorem \ref{thm:cono}, we deduce that the third fundamental form at the 
general point is given by a perfect cube.
\end{ex}

Let us see the following consequence of the preceding result. 
\begin{cor}
If the $s$-th Gauss map $\gamma^s$ ($s\ge 1$) for a $k$-dimensional variety $V$ has fibres of dimension $m$, then 
\begin{equation}\label{eq:dim}
d_{s+1}\le d_s+\binom{k-m+s}{s+1},
\end{equation}
and the equality is reached if and only if the $(s+1)$-th 
fundamental form $\abs{I^{s+1}}$ is given by the linear system of all the cones over a $\bP^{m-1}\subset \bP T_P(V)$.
\end{cor}

\begin{proof}
By Proposition \ref{cor_dimensione ff} we have to estimate the dimension $\Delta_{s+1}$ of the linear system given by the $(s+1)$-th 
fundamental form, but this, by Theorem \ref{thm:cono}, is formed by degree $(s+1)$ hypersurfaces which are cones over a 
$\bP^{m-1}$ in a  $\bP^{k-1}(= \bP (T_P(V)))$.  
Since the linear systems of these cones has (projective) dimension $\binom{k-m+s}{s+1}-1$ (is equal to the complete linear system of 
degree $(s+1)$ hypersurfaces in a $\bP^{k-m-2}$), we have the assertion.  
\end{proof}

\begin{ex}\label{ex:1}
Clearly, if $m=k$, then  $\gamma^s$ is constant, $\abs{I^{s+1}}=\emptyset$ and the variety is contained in a $\bP^{d_s}$. 
\end{ex}

\subsection{Varieties with a curve as the image of the $s$-th Gauss map}\label{ss:3}
Let us start with the following

\begin{lem}\label{lem:ff}
If $d_{s+1}=d_s+1$, then $%(d_s+1\le)
d_{s+2}\le d_s+2$. Moreover, if $d_{s+2}= d_s+2%=d_{s+1}+1
$, then $\abs{I^{s+1}}$ and $\abs{I^{s+2}}$ 
are generated by powers 
of the same linear form% and vice versa
, and if
 $d_s+1=d_{s+2}%(< d_s+2)
$, then our variety $V$ is contained in 
$\bP^{d_s+2}$% and vice versa
. 
\end{lem}

\begin{proof}
We apply Proposition \ref{cor_dimensione ff} to deduce that the  $(s+1)$-th fundamental form $\abs{I^{s+1}}$ is generated by only one form. 
Call this form $f$; by Corollary \ref{cor:jac}, if $g$ is a form in $\abs{I^{s+2}}$, then all its partial derivatives are in 
$\abs{I^{s+1}}$; therefore, by Euler's formula 
\begin{equation*}
g=\frac{1}{s+2}\sum_{i=1}^k\omega_i \frac{\partial g}{\partial \omega_i}=\frac{1}{s+2}\sum_{i=1}^k\omega_i a_if
\end{equation*}
$a_i\in \CC$, $i=1,\dotsc,k$, and we deduce that either $g$ is zero or $g=f\ell$, where $\ell$ is a linear form. 
If $g=f\ell$, by a change of coordinates we can suppose that $\ell=\omega_1$; then, we have
\begin{align*}
\frac{\partial g}{\partial \omega_1}&=f+\omega_1\frac{\partial f}{\partial \omega_1}\in\abs{I^{s+1}}=(f)  & \\
\frac{\partial g}{\partial \omega_j}&=\omega_1\frac{\partial f}{\partial \omega_j}\in\abs{I^{s+1}}=(f) & j=2,\dotsc,k\\
\end{align*}
which imply that either $f$ does not depend on $\omega_1$, but then $g$ is the zero form, or $f$ is reducible of the form 
$f=\omega_1h$. We can now proceed by induction to prove that $f=\omega_1^{s+1}$: if $f=\omega_1^{\ell}\phi$, then 
\begin{align*}
\frac{\partial g}{\partial \omega_1}&=(\ell+1)\omega^{\ell}\phi+\omega_1^{\ell+1}\frac{\partial \phi}{\partial \omega_1}=n \omega_1^{\ell}\phi & \\
\frac{\partial g}{\partial \omega_j}&=\omega_1^{\ell+1}\frac{\partial \phi}{\partial \omega_j}=n_j \omega_1^{\ell}\phi  & j=2,\dotsc,k\\
\end{align*}
$n,n_j\in \CC$, $j=1,\dotsc,k$, and, as above,  either $\phi$ does not depend on $\omega_1$, but then $g$ is the zero form, 
or $\phi$ is reducible of the form 
$\phi=\omega_1\phi'$.

If $g=0$, then we can conclude from Theorem \ref{thm:cono} (or, better, from  Example \ref{ex:1}). 
\end{proof}

\begin{cor}\label{cor:coro}
If $d_{s+2}=d_{s+1}+1=d_s+2$, then the  $s$-th and $(s+1)$-th Gauss maps have fibres of dimension $k-1$.
\end{cor}

\begin{proof}
By Lemma \ref{lem:ff} we have that the $\abs{I^{s+1}}$ and $\abs{I^{s+2}}$ are cones over a $\bP^{k-2}$, therefore,
by Theorem \ref{thm:cono}, the thesis follows.
\end{proof}

The case we are interested in now is the one with 
$m=k-1$ (i. e. the next after Example \ref{ex:1}), 
therefore, by \eqref{eq:dim}, $(d_s\le)d_{s+1}\le d_s+1$; if $d_{s+1}<d_s+1$, then  $d_{s+1}=d_s$, but in this case 
we would have no elements in   $\abs{I^{s+1}}$, and therefore we would be in Example \ref{ex:1}. We deduce that $d_{s+1}= d_s+1$, and 
we are in the case of Lemma \ref{lem:ff}, and by Theorem  \ref{thm:cono} the image of the  $s$-th  Gauss map has dimension $1$; 
the complete description of such cases will follow from \cite[(2.4)]{G-H}. To use this result, we need some notations and result: 
we start by recalling the notation introduced in \cite[\S 2(a)]{G-H}; let $B$ be an $r$-dimensional variety and let 
\begin{equation*}
f\colon B \to \GG(d,N)
\end{equation*}
be a morphism; let $y\in B$ be a general point, and let $S_y\subset V$ (where $\bP^N=\bP(V)$) be the $(d+1)$-dimensional vector space which 
corresponds to $f(y)$. Then, the differential of $f$ in $y$ can be thought of 
\begin{equation*}
    d_y f \colon T_y B \to  \Hom(S_y, N_y),
\end{equation*}
where $N_y:=\frac{V}{S_y}$. More explicitly, if $y_1,\dotsc, y_r$ are local coordinates of $B$ near $y$, and if $e_0(y), \dotsc e_d(y)$ is 
a basis for the the $(d+1)$-dimensional vector spaces near  $S_y$, we have 
\begin{equation*}
    d_y f\left(\frac{\partial }{\partial y_i}\right)(e_j(y))\equiv \frac{\partial (e_j(y))}{\partial y_i}\mod S_y 
\end{equation*}
for $i=1,\dotsc,r$ and $j=0,\dotsc d$. Then, fixed $w\in  T_y B$, one can define the ``infinitely near'' space 
\begin{equation*}
\frac{S_y}{dw}:=\im d_y f(w)\subset N_y,
\end{equation*}
and if we denote by $[v]$ the class of $v\in V$ in $N_y$, we consider the following subspaces of $V$:   
\begin{align*}
S_y+\frac{S_y}{dw}&:=\left\{v\in V\mid v=s+t, s\in S_y, [t]\in \frac{S_y}{dw}\right\}\\
S_y\cap\frac{S_y}{dw}&:=\left\{s\in S_y\mid ([0]=)[s]\in  \frac{S_y}{dw}\right\}=\ker (d_y f(w)). 
\end{align*}
Obviously, we have
\begin{align*}
\rk(d_yf)&=\rho& &\iff &\dim(S_y+\frac{S_y}{dw})&=d+1+\rho& &\iff &\dim(S_y\cap\frac{S_y}{dw})&=d+1-\rho.
\end{align*}
We also denote the projective subspaces of $\bP^N$ associated to the vector spaces just defined by 
\begin{align*}
\bP_y^d+\frac{\bP_y^d}{dw}&:=\bP(S_y+\frac{S_y}{dw})\\
\bP_y^d\cap\frac{\bP_y^d}{dw}&:=\bP(S_y\cap\frac{S_y}{dw}). 
\end{align*}

The case of $r=1$ is completely solved in \cite{G-H}; as in the cited article, with 
abuse of notation we denote by $y$ the local coordinate of $B$ near our general point $y$:
\begin{prop}{\cite[(2.4)]{G-H}}\label{prop:ghc}
With notations as above, if $r=1$ and  $\dim(\bP_y^d\cap\frac{\bP_y^d}{dw})=d-\rho$, 
then there exist $\rho$ curves in $\bP^N$, 
 $C_1, \dotsc, C_\rho$ parametrised by $y$, $a_1,\dotsc,a_\rho\in \NN$, and a fixed  $(d-\sum_{i=1}^\rho a_i)$-dimensional  
subspace $H\subset \bP^N$ such that $\bP_y^d$ is the span of $H$ together with the $(a_i-1)$-osculating $(a_i-1)$-planes 
 ${\tilde T}^{(a_i-1)}_P(C_i)\subset\bP^N$, that is 
\begin{equation*}
\bP_y^d=\gen{{\tilde T}^{(a_1-1)}_P(C_1),\dotsc, {\tilde T}^{(a_\rho-1)}_P(C_\rho), H}.
\end{equation*}
\end{prop}

Let us now use this result in our case: 
we take $B$ as the image of the $s$-th Gauss map:
\begin{equation*}
\gamma^s\colon V \to \im(\gamma^s)=:B \overset{i}\hookrightarrow \GG(d_s,N),
\end{equation*}
$\dim B=1$ and $S_y$ denotes just the $d_s$-dimensional vector subspace ${\tilde T}_P^{(s)}(V)\subset \CC^{N+1}$, 
where $y=\gamma^s(P)$, defined 
by---with notations as in the proof of Theorem \ref{thm:cono}---$S_y:=\gen{A_0,\dotsc,A_{d_s}}$. %, i. e. by \eqref{eq:deffo}. 
Moreover, 
by \eqref{eq:frobenio}, we have that $B$ defines also a ruled variety, which is indeed the variety of  osculating $s$-spaces to $V$:  
${\Tan}^s(V):=\cup_{y\in B}\bP_y^{d_s}$ of dimension $d_s+1$. 
By Corollary \ref{cor:coro}, if  ${\tilde T}_P^{(s)}(V)\neq {\tilde T}_P^{(s+1)}(V)$, the same happens with  ${\tilde T}_P^{(s+1)}(V)$:
we have, on  $B$,  the  $S_y'$'s, which are the $d_{s+1}$-dimensional vector subspaces  ${\tilde T}_P^{(s+1)}(V)\subset \CC^{N+1}$, 
they define the (ruled) variety of  osculating $(s+1)$-spaces to $V$
${\Tan}^{s+1}(V):=\cup_{y\in B}\bP_y^{d_{s+1}}$ of dimension $d_s+2$ 
such that $V\subset {\Tan}^s(V) \subset {\Tan}^{s+1}(V)$. 

Since, by definition of the $(s+1)$ osculating space, we have that, for $y\in B$,   
\begin{equation*}
\bP_y^{d_s}+\frac{\bP_y^{d_s}}{dy}\subset \bP_y^{d_{s+1}},
\end{equation*}
which implies
\begin{equation*}
\dim(\bP_y^{d_s}\cap\frac{\bP_y^{d_s}}{dy})=d_s-1.
\end{equation*}
Then, by Proposition \ref{prop:ghc}
\begin{prop}
If the $s$-th Gauss map has dimension one, then 
there exists a curve $C\subset\bP^N$ having $(a-1)$-osculating $(a-1)$ planes $\bP_y^{a-1}:={\tilde T}_y^{(a-1)}(C)$ and a fixed 
 $\bP^{N-a-1}$ such that 
\begin{equation*}
\bP_y^{d_s}=\bP_y^{a-1}+\bP^{N-a-1}.
\end{equation*}
In other words,  ${\Tan}^s(V)$, is the cone over  $\bP^{N-a-1}$ of  ${\Tan}^{a-1}(C)$ (and the same holds true for 
higher Gauss maps).
\end{prop}

Clearly, one can deduce more general results on the varieties of  osculating $s$-spaces in the same vein as at the end of \cite[\S 3]{G-H}.

%\giallo{qui siamo arrivati}

%%%%%%%%%%%%%%%%%%%%%%%%%%%%%%%%%%%%%
%%%%%%%%%%%%%%%%%%%%%%%%%%%%%%%%%%%%%
%%%%%%%%%%%%%%%%%%%%%%%%%%%%%%%%%%%%%


\begin{thebibliography}{99}

\bibitem{AG} M. A. Akivis \and V. V. Goldberg, \emph{Differential Geometry of Varieties with Degenerate Gauss Maps},
CMS Books in Mathematics/Ouvrages de Math\'ematiques de la SMC, 18. Springer-Verlag, New York, 2004.

\bibitem{B}
E. Bertini,
\emph{Introduzione alla geometria projettiva degli iperspazi con appendice sulle curve algebriche e loro singolarit\`a}, 
E. Spoerri, Pisa, 1907.

\bibitem{C}
M. Castellani,
\emph{Sulle superficie i cui spazi osculatori sono biosculatori}, 
Rom. Acc. L. Rend. (5) \textbf{31}  (1922), no. 1,  347--350.


% \bibitem{De P-I} A. De Paris \and G. Ilardi, \emph{Some formulae arising in projective-differential geometry},  
% Ann. Univ. Ferrara, Nuova Ser., Sez. VII, \textbf{47} (2001), 63--88.


\bibitem{De Poi- Di G-I} P. De Poi, R. Di Gennaro \and G. Ilardi, \emph{On varieties with higher osculating spaces}, 
Rev. Mat. Iberoam. \textbf{29} (2013), no. 4, 191--1210.

\bibitem{Delp}
P. del Pezzo,
 \emph{Sugli spazii tangenti ad una superficie o ad una variet\`a immersa in uno spazio di pi\'u dimensioni}, 
Nap. Rend. \textbf{XXV} (1886), 176--180. 

\bibitem{dir}
S. Di Rocco, K. Jabbusch and A. Lundman, \emph{A note on higher order Gauss maps}, 
\url{arXiv:1410.4811 [math.AG]}

\bibitem{F-I} D. Franco \and G. Ilardi, \emph{On multiosculating spaces}, 
Commun. Algebra, \textbf{29} (2001), no. 7, 2961--2976 .

% \bibitem{G-Hli} Griffiths Ph. \and Harris J., \emph{Principles of algebraic
% geometry}, Wiley Interscience (1978).

\bibitem{G-H} P. Griffiths and J. Harris \emph{Algebraic geometry and local differential geometry},
Ann. Sci. \'Ecole Norm. Sup. (4), \textbf{12} (1979), no. 3, 355--452.

\bibitem{H} R. Hartshorne, \newblock \emph{Algebraic Geometry}, \newblock Number~52 in Graduate Text in Mathematics. Springer-Verlag, New York, 1977.

\bibitem{I}
G. Ilardi, \emph{Togliatti systems}, Osaka J. Math. \textbf{43} (2006), no. 1, 1--12.

% \bibitem{K} S. L. Kleiman, \emph{The enumerative theory of singularities},
% in: Real and Complex Singularities, P. Holme ed.,
% Proc. Nordic Summer Sch., Symp. Math., Oslo 1976, 297--396, 1977.

% \bibitem{M-T} E. Mezzetti \and O. Tommasi, \emph{On projective varieties of dimension $n+k$
% covered by $k$ spaces}, 
% Ill. J. Math., \textbf{46} (2002), No. 2, 443--465.

\bibitem{L}
J. M. Landsberg, \emph{On second fundamental forms of projective varieties}, Invent. Math. \textbf{117} (1994), no. 2, 303--315.


\bibitem{RP}
R. Piene,
\emph{A note on higher order dual varieties, with an application to scrolls}, Singularities, Part 2
(Arcata, Calif., 1981), 1983, pp. 335--342.

\bibitem{S}
C. Segre, 
\emph{Preliminari di una teoria delle variet\`a luoghi di spazi}, 
Palermo Rend. \textbf{30}  (1910), 87--121.

\bibitem{T}
E. Togliatti, 
\emph{Alcuni esemp\^\i{} di superficie algebriche degli iperspaz\^\i{} che rappresentano un'equazione di Laplace}, 
Comm. Math. Helvetici \textbf{1} (1929), 255--272.


\bibitem{Z}
F. L. Zak,
\emph{Tangents and secants of algebraic varieties}, Translations of Mathematical Monographs, vol. 127,
American Mathematical Society, Providence, RI, 1993. Translated from the Russian manuscript by the
author.

\end{thebibliography}
\end{document}